\documentclass[a4paper, intlimits, reqno]{amsart}

\usepackage[english]{babel}
\usepackage[T1]{fontenc}
\usepackage[utf8]{inputenc}

\usepackage{amsmath}
\usepackage{amssymb}
\usepackage{MnSymbol}
\usepackage{amsthm}
\allowdisplaybreaks
\usepackage{amsfonts}
\usepackage{mathrsfs} 
\usepackage{mathtools}
\usepackage[all]{xy}
\usepackage{nicefrac}
\usepackage{enumitem}

\usepackage{multicol}
\usepackage{url}
\usepackage{dsfont}
\usepackage[numbers,sort&compress]{natbib}
\usepackage{doi}
\usepackage{prettyref}
\usepackage{xcolor}
\usepackage{orcidlink}

\newrefformat{defn}{Definition \ref{#1}}
\newrefformat{rem}{Remark \ref{#1}}
\newrefformat{sect}{Section \ref{#1}}
\newrefformat{sub}{Section \ref{#1}}
\newrefformat{prop}{Proposition \ref{#1}}
\newrefformat{thm}{Theorem \ref{#1}}
\newrefformat{cor}{Corollary \ref{#1}}
\newrefformat{ex}{Example \ref{#1}}

\swapnumbers
\newtheoremstyle{dotless}{}{}{\itshape}{}{\bfseries}{}{}{}
\theoremstyle{dotless}
\theoremstyle{plain}
\newtheorem{thm}{Theorem}[section]

\newtheorem{prop}[thm]{Proposition}
\newtheorem{cor}[thm]{Corollary}
\theoremstyle{definition}
\newtheorem{defn}[thm]{Definition}
\newtheorem{rem}[thm]{Remark}
\newtheorem{exa}[thm]{Example}
\newcommand{\N} {\mathbb{N}}

\newcommand{\R} {\mathbb{R}}
\newcommand{\C} {\mathbb{C}}
\newcommand{\K} {\mathbb{K}}
\newcommand{\D} {\mathbb{D}}
\newcommand{\F} {\mathcal{F}(\Omega)}
\newcommand{\FE} {\mathcal{F}(\Omega,E)}

\DeclareMathOperator{\id}{id}
\providecommand{\differential}{\mathrm{d}}
\renewcommand{\d}{\differential}

\newcommand{\vertiii}[1]{{\left\vert\kern-0.25ex\left\vert\kern-0.25ex\left\vert #1 
    \right\vert\kern-0.25ex\right\vert\kern-0.25ex\right\vert}}

\makeatletter
\newcommand{\fakephantomsection}{%
  \Hy@GlobalStepCount\Hy@linkcounter%
  \Hy@MakeCurrentHref{\@currenvir.\the\Hy@linkcounter}
  \Hy@raisedlink{\hyper@anchorstart{\@currentHref}\hyper@anchorend}%
}
\makeatother

\begin{document}

\title[On isometric linearisation, existence and uniqueness of preduals]{On linearisation, existence and uniqueness of preduals: The isometric case}
\author[K.~Kruse]{Karsten Kruse\,\orcidlink{0000-0003-1864-4915}}
\address[Karsten Kruse]{University of Twente, Department of Applied Mathematics, P.O. Box 217, 7500 AE Enschede, The Netherlands, and Hamburg University of Technology, Institute of Mathematics, Am Schwarzenberg-Campus~3, 21073 Hamburg, Germany}

\email{k.kruse@utwente.nl}



\date{\today}
\begin{abstract}
We study the problem of existence and uniqueness of isometric Banach preduals of a Banach space. 
We derive necessary and sufficient conditions for the existence of an isometric Banach predual of a Banach space $X$. 
Then we focus on the case that $X=\F$ is a Banach space of scalar-valued functions on a non-empty set $\Omega$ and 
describe those spaces which admit a special isometric Banach predual, namely a \emph{strong isometric Banach linearisation}, 
i.e.~there are a Banach space $Y$, a map $\delta\colon\Omega\to Y$ and an isometric isomorphism 
$T\colon\F\to Y^{\ast}$ such that $T(f)\circ \delta= f$ for all $f\in\F$. Finally, we give necessary and sufficient 
conditions for Banach spaces $\F$ with a strong isometric Banach linearisation to have a (strongly) unique isometric Banach predual. 
\end{abstract}
\maketitle

\section{Introduction}\label{sect:intro}

This paper centres on the problem of giving necessary and sufficient conditions for the existence and uniquenes 
of an isometric Banach predual of a Banach space. An \emph{isometric Banach predual} of a Banach space $X$ is a tuple 
$(Y,\varphi)$ of a Banach space $Y$ and an isometric isomorphism $\varphi\colon X\to Y^{\ast}$ where $Y^{\ast}$ is the dual space 
of $Y$ equipped with the dual norm. 
The Dixmier--Ng theorem \cite[Theorem 1, p.~279]{ng1971} says that a Banach space $(X,\|\cdot\|)$ has an isometric Banach predual 
if there exists a locally convex Hausdorff topology $\tau$ on $X$ such that the $\|\cdot\|$-closed unit ball $B_{\|\cdot\|}$ is 
$\tau$-compact. Other variants of this theorem are due to Dixmier \cite[Th\'{e}or\`{e}me 19, p.~1069]{dixmier1948}, 
Waelbroeck \cite[Proposition 1, p.~122]{waelbroeck1966} and Kaijser \cite[Theorem 1, p.~325]{kaijser1977}.  

Being familiar with the theory of Saks spaces \cite{cooper1978} and the mixed topology introduced by Wiweger \cite{wiweger1961}, 
we recognize that the condition that $B_{\|\cdot\|}$ is $\tau$-compact means the that the triple $(X,\|\cdot\|,\tau)$ is a 
\emph{semi-Montel Saks space}, i.e.~$(X,\|\cdot\|,\tau)$ is a Saks space and $(X,\gamma)$ a semi-Montel space where 
$\gamma\coloneqq \gamma(\|\cdot\|,\tau)$ denotes the mixed topology. 
As a first step, we show in \prettyref{prop:dixmier_ng_mixed} that the a priori weaker condition that $(X,\|\cdot\|,\tau)$ 
is a semi-reflexive Saks space is already sufficient to guarantee that the Banach space $X$ has an isometric Banach predual, 
even an isometric prebidual (see \prettyref{defn:isom_prebidual}). Then we show in \prettyref{cor:existence_isom_banach_predual} 
that this condition is also necessary and equivalent to the compactness of $B_{\|\cdot\|}$ w.r.t.~to a coarser locally 
convex Hausdorff topology. This allows us discover many Banach spaces with an isometric Banach predual in 
\prettyref{ex:subspace_cont_mixed}. 

The examples from \prettyref{ex:subspace_cont_mixed} have another thing in common. They are Banach spaces $\F$ 
of scalar-valued functions on a non-empty set $\Omega$. For such Banach spaces we are interested in the question whether they have a 
(strongly) unique isometric Banach predual. A Banach space $X$ with an isometric Banach predual is said to have a 
\emph{strongly unique isometric Banach predual} if for all isometric Banach preduals $(Y,\varphi)$ and $(Z,\psi)$ of $X$ 
and all isometric isomorphisms $\alpha\colon Z^{\ast}\to Y^{\ast}$ there is an isometric isomorphism $\lambda\colon Y\to Z$ such that 
$\lambda^{t}=\alpha$. If this is fulfilled without the requirement that $\lambda^{t}=\alpha$, then $X$ is said to have 
a \emph{unique isometric Banach predual}, and it is an open problem whether both notions of uniqueness are equivalent, 
see \cite[Problem (2), p.~186]{godefroy1989}. 
The question whether a Banach space has a (strongly) unique isometric Banach predual got a lot of attention. 
We refer the reader in the general setting of Banach spaces to the paper by Brown and Ito \cite{brown1980} 
and the nice detailed survey by Godefroy \cite{godefroy1989}. In particular, examples of Banach spaces with a 
strongly unique isometric Banach predual are: von Neumann algebras by Sakai \cite[1.13.3 Corollary, p.~30]{sakai1998}, 
the space $H^{\infty}$ of bounded holomorphic functions on the complex open unit disc by Ando \cite[Theorem 1, p.~34]{ando1978}, 
the space of Bloch functions that vanish at the origin by Nara \cite[Corollary 2, p.~99]{nara1990}, 
biduals of non-reflexive Banach spaces which are $M$-ideals in their bidual by Harmand, Werner and Werner 
\cite[Proposition 2.10, p.~122]{harmand1993}, Banach spaces with isometric separable $L$-embedded Banach predual by 
Pfitzner \cite[Theorems 2, 3, p.~1040]{pfitzner2007}, the non-commutative analytic Toeplitz algebras $\mathfrak{L}_{n}$, $n\geq 2$, 
and every free semigroup algebra by Davidson and Wright \cite[Theorems 3.3, 3.5, p.~416--417]{davidson2011}, 
certain weak-star closed linear subspaces of the space $L(X)$ of continuous linear operators from $X$ to $X$ for a separable 
reflexive Banach space $X$ by Godefroy \cite[Theorem 2.1, p.~811]{godefroy2014}, 
the space $L(X,Y)$ of continuous linear operators between to Banach spaces $X$ and $Y$ by Gardella and Thiel 
\cite[Corollary 5.6, p.~18]{gardella2020} if $Y$ has a strongly unique isometric Banach predual, 
and certain spaces of scalar-valued Lipschitz continuous functions by Weaver \cite[Theorems 3.2, 3.3, p.~471--472]{weaver2018a}.\footnote{See the footnote on page 17.}

We give a characterisation of Banach spaces $\F$ of scalar-valued functions on a non-empty set $\Omega$ 
having a (strongly) unique isometric Banach predual in \prettyref{cor:strongly_unique_isom_predual_without_lin_ind} 
and \prettyref{cor:unique_isom_predual_without_lin_ind}. We derive our characterisation by using 
strong isometric Banach linearisations of $\F$ and the connection between strong (uniqueness) and an equivalence 
relation (see \prettyref{defn:isom_predual_equivalent}) on the family of isometric Banach preduals of a Banach space, 
see \prettyref{prop:isom_unique_equivalent} and \prettyref{prop:isom_strongly_unique_equivalent}. 
A \emph{strong isometric Banach linearisation} of a Banach space $\F$ is a triple $(\delta,Y,T)$ of a Banach space 
$Y$, a map $\delta\colon\Omega\to Y$ and an isometric isomorphism $T\colon\F\to Y^{\ast}$ such that 
$T(f)\circ \delta= f$ for all $f\in\F$. In particular, if $\F$ admits a strong isometric Banach linearisation, 
then $(Y,T)$ is an isometric Banach predual of $\F$. A classical example is the (completion of the) 
projective tensor product $Y\coloneqq F\widehat{\otimes}_{\pi} G$ of two Banach spaces $F$ and $G$ which induces a strong isometric 
Banach linearisation of the Banach space $\mathscr{L}(F\times G)$ of continuous bilinear forms on $\Omega\coloneqq F\times G$ 
via the map given by $\delta(x,y)(B)\coloneqq B(x,y)$ for $(x,y)\in F\times G$ and $B\in\mathscr{L}(F\times G)$, 
see e.g.~Ryan \cite[Theorem 2.9, p.~22]{ryan2002}. Using the Dixmier--Ng theorem, 
other special cases of isometric Banach linearisations of weighted Banach spaces 
of holomorphic or harmonic functions were derived by Mujica \cite[2.1 Theorem, p.~869]{mujica1991}, 
Bonet, Doma\'nski and Lindstr\"om \cite[(c), p.~243]{bonet2001}, 
Laitila and Tylli \cite[(iv), p.~13]{laitila2006}, 
Beltr\'an \cite[Remark 4, p.~284]{beltran2012} and \cite[p.~67]{beltran2014}, 
Jord\'a \cite[Proposition 6, p.~3]{jorda2013}, 
Gupta and Baweja \cite[Theorem 3.1 (Linearization Theorem), p.~128]{gupta2016}, 
Quang \cite[Theorem 3.5 (Linearization), p.~19]{quang2023}, 
and Aron, Dimant, Garc\'{i}a-Lirola and Maestre \cite[Proposition 2.3 (c), p.~3029]{aron2024}. 
We show that a Banach space $(\F,\|\cdot\|)$ of scalar-valued functions on a non-empty set $\Omega$ admits a 
strong isometric Banach linearisation if and only if there is a locally convex Hausdorff topology $\tau$ on $\F$ such that 
$B_{\|\cdot\|}$ is $\tau$-compact and $\tau$ finer than the topology of pointwise convergence. We also give several other 
necessary and sufficient conditions for the existence of a strong isometric Banach linearisation 
in \prettyref{thm:existence_siB_linearisation}. Besides the application to the question of (strong) uniqueness, 
we give another application of strong isometric Banach linearisations to weakly compact composition 
operators in \prettyref{thm:compact_operator}. We refer the reader who is also interested in the corresponding results 
of the present paper in the non-isometric Banach setting or even in the locally convex setting to \cite{kruse2024,kruse2025}.

\section{Notions and preliminaries}
\label{sect:notions}

In this short section we recall some basic notions and results from \cite[Sections 2, 3]{kruse2024,kruse2025} 
and present some preliminary results on dual Banach spaces and their Banach preduals. 
For a locally convex Hausdorff space $X$ over the field $\K\coloneqq\R$ or $\C$ we denote by $X'$ the topological linear dual space 
of $X$. If we want to emphasize the dependency on the locally convex Hausdorff topology $\tau$ of $X$, we write $(X,\tau)$ and 
$(X,\tau)'$ instead of just $X$ and $X'$, respectively. If $(X,\|\cdot\|)$ is a normed space, we write $\tau_{\|\cdot\|}$ 
for the locally convex Hausdorff topology induced by $\|\cdot\|$ and set $X^{\ast}\coloneqq (X,\|\cdot\|)^{\ast}\coloneqq 
(X,\tau_{\|\cdot\|})'$. We always equip $X^{\ast}$ with the dual norm 
given by $\|x^{\ast}\|_{X^{\ast}}\coloneqq \sup_{x\in B_{\|\cdot\|}}|x^{\ast}(x)|$ for $x^{\ast}\in X^{\ast}$ where 
$B_{\|\cdot\|}\coloneqq\{x\in X\;|\;\|x\|\leq 1\}$ denotes the $\|\cdot\|$-closed unit ball of $X$. 
Further, for a linear subspace $V\subset X'$ of the dual of a locally convex Hausdorff space $X$ we denote by 
$\sigma(X,V)$ the locally convex topology on $X$ induced by the system of seminorms given by
\[
p_{N}(x)\coloneqq\sup_{y\in N}|y(x)|,\quad x\in X,
\]
for finite sets $N\subset V$. We note that $\sigma(X,V)$ is Hausdorff if and only if $V$ separates the points of $X$. 
Furthermore, for a continuous linear map $T\colon X\to Y$ between two locally convex Hausdorff spaces $X$ and $Y$ we denote by 
$T^{t}\colon Y'\to X'$, $y'\mapsto y'\circ T$, the \emph{dual map} of $T$. 
Moreover, for two locally convex Hausdorff topologies $\tau_{0}$ and $\tau_{1}$ 
on $X$ we write $\tau_{0}\leq\tau_{1}$ if $\tau_{0}$ is coarser than $\tau_{1}$. 
We write $\tau_{\operatorname{co}}$ for the \emph{compact-open topology}, i.e.~the topology of uniform convergence on compact 
subsets of $\Omega$, on the space $\mathcal{C}(\Omega)$ of $\K$-valued continuous functions on a topological Hausdorff space $\Omega$. 
In addition, we write $\tau_{\operatorname{p}}$ for the \emph{topology of pointwise convergence} on the space $\K^{\Omega}$ of $\K$-valued functions on a set $\Omega$. By a slight abuse of notation we also use the symbols $\tau_{\operatorname{co}}$ and 
$\tau_{\operatorname{p}}$ for the relative compact-open topology and the relative topology of pointwise convergence on topological subspaces of $\mathcal{C}(\Omega)$ and $\K^{\Omega}$, respectively. 

\begin{defn}[{\cite[Definition 2.1]{gardella2020}}]\label{defn:isom_predual_equivalent}
Let $X$ be a normed space. 
\begin{enumerate}
\item[(a)] We call $X$ a \emph{(isometric) dual Banach space} if there are a Banach space $Y$ and a topological 
(isometric) isomorphism $\varphi\colon X\to Y^{\ast}$. The tuple $(Y,\varphi)$ is called a \emph{(isometric) Banach predual} of $X$. 
\item[(b)] Let $X$ be a (isometric) dual Banach space. We say that two (isometric) Banach preduals $(Y,\varphi)$ and $(Z,\psi)$ of $X$ are \emph{(isometrically) equivalent} if there is a topological (isometric) isomorphism 
$\lambda\colon Y\to Z$ such that $\lambda^{t}=\varphi\circ\psi^{-1}$. 
\end{enumerate} 
\end{defn}

\prettyref{defn:isom_predual_equivalent} (a) is already given in e.g.~\cite[p.~321]{brown1975} 
in the isomorphic setting and in \cite[p.~132]{godefroy1989} in isometric setting, too. 
We note that a normed space that is a dual Banach space is already a Banach space itself. 
Due to \cite[Theorem, p.~487]{davis1973} there are dual Banach spaces which are not isometric dual Banach spaces. 
Namely, if $(X,\|\cdot\|)$ is a non-reflexive isometric dual Banach space, 
e.g.~the space $(\ell^{\infty},\|\cdot\|_{\infty})$ of bounded sequences, then there 
exists an equivalent norm $\vertiii{\cdot}$ on $X$ such that $(X,\vertiii{\cdot})$ is not an isometric dual Banach space.

\begin{rem}\label{rem:dual_isom_map_isom}
Let $Y$ and $Z$ be Banach spaces and $\lambda\colon Y\to Z$ a topological isomorphism. 
If the dual map $\lambda^{t}\colon Z^{\ast}\to Y^{\ast}$ is an isometry, then $\lambda$ is also an isometry. 
Indeed, since $\lambda\colon Y\to Z$ is a topological isomorphism and $\lambda^{t}$ an isometry, we obtain that 
$\lambda^{t}$ is an isometric isomorphism, implying $\lambda^{t}(B_{\|\cdot\|_{Z^{\ast}}})=B_{\|\cdot\|_{Y^{\ast}}}$ and
\[
 \|\lambda(y)\|_{Z}
=\sup_{z^{\ast}\in B_{\|\cdot\|_{Z^{\ast}}}}|z^{\ast}(\lambda(y))|
=\sup_{z^{\ast}\in B_{\|\cdot\|_{Z^{\ast}}}}|\lambda^{t}(z^{\ast})(y)|
=\sup_{y^{\ast}\in B_{\|\cdot\|_{Y^{\ast}}}}|y^{\ast}(y)|\\
=\|y\|_{Y}
\]
for all $y\in Y$. Hence $\lambda$ is an isometry. 
\end{rem}

Due to the preceding remark we may equivalently replace the isometric isomorphism $\lambda\colon Y\to Z$ in 
\prettyref{defn:isom_predual_equivalent}(b) by a topological isomorphism because it is 
automatically isometric if $(Y,\varphi)$ and $(Z,\psi)$ are isometric Banach preduals of $X$ 
(cf.~\cite[Remark 2.4]{gardella2020}).

\begin{prop}\label{prop:equiv_implies_isom_equiv}
Let $X$ be an isometric dual Banach space with isometric Banach preduals $(Y,\varphi)$ and $(Z,\psi)$. 
Then the following assertions are equivalent.
\begin{enumerate}
\item[(a)] $(Y,\varphi)$ and $(Z,\psi)$ are equivalent.  
\item[(b)] $(Y,\varphi)$ and $(Z,\psi)$ are isometrically equivalent.  
\end{enumerate} 
\end{prop}

In addition, a (isometric) Banach predual $(Y,\varphi)$ of a (isometric) dual Banach space $X$ may be considered 
as a closed subspace of $X^{\ast}$.

\begin{rem}\label{rem:isom_predual_into_dual}
Let $X$ be a (isometric) dual Banach space with (isometric) Banach predual $(Y,\varphi)$. Then 
$\Phi_{\varphi}\colon Y\to X^{\ast}$, $y\mapsto [x\mapsto \varphi(x)(y)]$, 
is an injective continuous (isometric) linear map (see e.g.~\cite[Proposition 2.2, p.~1593]{kruse2024}).
\end{rem}

Next, we want to study whether a (isometric) dual Banach space has a \emph{unique} (isometric) Banach predual by identification via 
topological (isometric) isomorphisms.

\begin{defn}\label{defn:isom_unique_predual}
Let $X$ be a (isometric) dual Banach space.  
\begin{enumerate}
\item[(a)] We say that $X$ has a \emph{unique (isometric) Banach predual} if for all (isometric) Banach preduals $(Y,\varphi)$ and $(Z,\psi)$ of $X$ there is a topological (isometric) isomorphism $\lambda\colon Y\to Z$. 
\item[(b)] We say that $X$ has a \emph{strongly unique (isometric) Banach predual} if for all (isometric) Banach preduals 
$(Y,\varphi)$ and $(Z,\psi)$ of $X$ and all topological (isometric) isomorphisms $\alpha\colon Z^{\ast}\to Y^{\ast}$ 
there is a topological (isometric) isomorphism $\lambda\colon Y\to Z$ such that $\lambda^{t}=\alpha$. 
\end{enumerate} 
\end{defn}

\prettyref{defn:isom_unique_predual} (a) is already given in e.g.~\cite[p.~321]{brown1975} 
in the isomorphic setting and in \cite[p.~132]{godefroy1989} in the isometric setting.
\prettyref{defn:isom_unique_predual} (b) is given in an equivalent form in e.g.~\cite[p.~134]{godefroy1989} and \cite[p.~469]{weaver2018a} in the isometric setting. The next two propositions follow directly from the proofs of 
\cite[2.6, 2.7 Propositions, p.~958]{kruse2025}.

\begin{prop}\label{prop:isom_unique_equivalent}
Let $X$ be a (isometric) dual Banach space. Then the following assertions are equivalent.
\begin{enumerate}
\item[(a)] $X$ has a unique (isometric) Banach predual. 
\item[(b)] For all (isometric) Banach preduals $(Y,\varphi)$ and $(Z,\psi)$ of $X$ there 
is a topological (isometric) isomorphism $\mu\colon X\to Z^{\ast}$ such that $(Y,\varphi)$ and $(Z,\mu)$ are 
(isometrically) equivalent. 
\end{enumerate}
\end{prop}

\begin{prop}\label{prop:isom_strongly_unique_equivalent}
Let $X$ be a (isometric) dual Banach space. Then the following assertions are equivalent.
\begin{enumerate}
\item[(a)] $X$ has a strongly unique (isometric) Banach predual. 
\item[(b)] All (isometric) Banach preduals of $X$ are (isometrically) equivalent.
\end{enumerate}
\end{prop}

\prettyref{prop:isom_strongly_unique_equivalent} (b) is used in \cite[Definition 2.7]{gardella2020} 
to give an equivalent definition of a (isometric) dual Banach space 
having a strongly unique (isometric) Banach predual.

\begin{rem}\label{rem:unique_predual_implies_unique_isom}
Let $X$ be an isometric dual Banach space. If $X$ has a (strongly) unique Banach predual, then $X$ has a (strongly) unique 
isometric Banach predual. Indeed, this follows from \prettyref{prop:equiv_implies_isom_equiv}, 
\prettyref{prop:isom_unique_equivalent} and \prettyref{prop:isom_strongly_unique_equivalent}.
\end{rem}

The converse of \prettyref{rem:unique_predual_implies_unique_isom} does not hold by 
\cite[Example 6.7]{gardella2020} where it is shown that the Banach space $L(\ell^{2})$ of continuous linear operators from 
$\ell^{2}$ to $\ell^{2}$ does not have a strongly unique Banach predual even though it has a strongly unique isometric Banach predual. 
Here, $\ell^{2}$ denotes the Banach space of absolutely square summable sequences with its usual norm. 
Now, let us introduce a special (isometric) Banach predual of a (isometric) dual Banach space of scalar-valued functions, 
whose general definition goes back to \cite[p.~181, 184, 187]{jaramillo2009} and \cite[p.~683]{carando2004} in the non-isometric setting 
(see \cite[Definition 2.3, p.~1593]{kruse2024} and \cite[Proposition 2.6, p.~1595]{kruse2024}). 

\begin{defn}\label{defn:isom_linearisation}
Let $\F$ be a Banach space of $\K$-valued functions on a non-empty set $\Omega$. 
\begin{enumerate}
\item[(a)] We call a triple $(\delta,Y,T)$ of a completely normable locally convex Hausdorff space $Y$ over the field $\K$, 
a map $\delta\colon\Omega\to Y$ and a topological isomorphism $T\colon\F\to Y'$ a 
\emph{strong Banach linearisation of} $\F$ if $T(f)\circ \delta= f$ for all $f\in\F$. 
\item[(b)] We call a triple $(\delta,Y,T)$ of a Banach space $Y$ over the field $\K$, 
a map $\delta\colon\Omega\to Y$ and an isometric isomorphism $T\colon\F\to Y^{\ast}$ a 
\emph{strong isometric Banach linearisation of} $\F$ if $T(f)\circ \delta= f$ for all $f\in\F$. 
\item[(c)] We say that $\F$ \emph{admits a strong (isometric) Banach linearisation} 
if there exists a strong (isometric) Banach linearisation $(\delta,Y,T)$ of $\F$.
\end{enumerate}
\end{defn}

Clearly, if $(\delta,Y,T)$ is a strong Banach linearisation of $\F$, then there is a norm $\|\cdot\|_{Y}$ on $Y$ such that 
$((Y,\|\cdot\|_{Y}), T)$ is a Banach predual of $\F$. Further, $(Y,T)$ is an isometric Banach predual of $\F$ if $(\delta,Y,T)$ 
is a strong isometric Banach linearisation of $\F$. However, there are (isometric) dual Banach spaces which do not admit a 
strong (isometric) Banach linearisation (cf.~\cite[p.~190]{jaramillo2009}). 
For instance, let $K$ be an infinite compact Hausdorff space. Then the Banach space 
$(\mathcal{C}(K),\|\cdot\|_{\infty})$ of $\R$-valued continuous functions on $K$ does not admit a strong Banach linearisation by 
\cite[Corollary 2.4, p.~190]{jaramillo2009}, in particular no strong isometric Banach linearisation. 
If $K$ is also hyperstonean, then $\mathcal{C}(K)$ is an isometric dual Banach space with strongly unique isometric Banach predual
by \cite[Corollaire, p.~171]{dixmier1955}, \cite[Th\'{e}or\`{e}me 2, p.~554--555]{grothendieck1955} and 
\cite[1.13.3 Corollary, p.~30]{sakai1998} (cf.~\cite[Theorem 6.4.2, p.~200]{dales2016}).

\begin{exa}\label{ex:isom_linearisation_of_l1}
Let $(\ell^{1},\|\cdot\|_{1})$ denote the Banach space of complex absolutely summable sequences on $\N$, 
$(c_{0},\|\cdot\|_{\infty})$ the Banach space of complex zero sequences 
and $(c,\|\cdot\|_{\infty})$ the Banach space of complex convergent sequences, all three equipped with their usual norms. 

(i) We define the isometric isomorphism 
\[
\varphi_{0}\colon (\ell^{1},\|\cdot\|_{1})\to (c_{0},\|\cdot\|_{\infty})^{\ast},\;\varphi_{0}(x)(y)\coloneqq \sum_{k=1}^{\infty}x_{k}y_{k},
\]
and $\delta\colon\N\to c_{0}$, $\delta(n)\coloneqq e_{n}$, where $e_{n}$ denotes the $n$-th unit sequence. Then $(\delta,c_{0},\varphi_{0})$ is a strong isometric Banach linearisation of $\ell^{1}$ 
by \cite[Example 2.4 (i), p.~1594]{kruse2024}. 

(ii) We define the topological, non-isometric isomorphism
\[
\psi\colon \;(\ell^{1},\|\cdot\|_{1})\to (c,\|\cdot\|_{\infty})^{\ast},\;\psi(x)(y)\coloneqq y_{\infty}x_{1}+\sum_{k=1}^{\infty}(y_{k}-y_{\infty})x_{k+1},
\]
where $y_{\infty}\coloneqq\lim_{n\to\infty}y_{n}$ for $y\in c$, and the map 
$\widetilde{\delta}\colon\N\to c$ given by $\widetilde{\delta}(1)\coloneqq (1,1,\ldots)$ and 
$\widetilde{\delta}(n)\coloneqq e_{n-1}$ for all $n\geq 2$.
Then $(\widetilde{\delta},c,\psi)$ is a strong non-isometric Banach linearisation of $\ell^{1}$ 
by \cite[Example 2.4 (iii), p.~1594]{kruse2024}. 
\end{exa}

Further, there is an isometric Banach predual $(c,\varphi_{c})$ of $\ell^{1}$ which cannot be augmented by a map 
$\delta\colon\N\to c$ to a strong isometric Banach linearisation of $\ell^{1}$ by \cite[Example 2.4 (ii), p.~1594]{kruse2024}. 
However, we will give a characterisation of those isometric Banach preduals for which this is possible 
in \prettyref{thm:isometric_sB_linearisation_in_equivalence_class}.
Moreover, we have the following relation between strong isometric Banach linearisations and (strongly) 
unique isometric Banach preduals. The corresponding results for strong Banach linearisations 
and (strongly) unique Banach preduals are already proved in 
\cite[Proposition 3.8, p.~961]{kruse2025} and \cite[Corollary 3.9, p.~962]{kruse2025}. 

\begin{prop}\label{prop:strongly_unique_isom_predual}
Let $\F$ be a Banach space of $\K$-valued functions on a non-empty set $\Omega$ and $(\delta,Y,T)$ a strong 
isometric Banach linearisation of $\F$. Consider the following assertions.
\begin{enumerate}
\item[(a)] $\F$ has a strongly unique isometric Banach predual. 
\item[(b)] For every isometric Banach predual $(Z,\varphi)$ of $\F$ and every $x\in\Omega$ there is a (unique) $z_{x}\in Z$ with 
$T(\cdot)(\delta(x))=\varphi(\cdot)(z_{x})$.
\item[(c)] For every isometric Banach predual $(Z,\varphi)$ of $\F$ there is a (unique) map 
$\widetilde{\delta}\colon\Omega\to Z$ such that $(\widetilde{\delta},Z,\varphi)$ is a strong isometric Banach linearisation of $\F$.
\end{enumerate}
Then it holds that (a)$\Rightarrow$(b)$\Leftrightarrow$(c). If the family $(\delta(x))_{x\in\Omega}$ is 
linearly independent, then it holds that (b)$\Rightarrow$(a). 
\end{prop}
\begin{proof}
This statement follows directly from the proof of \cite[Proposition 3.8, p.~961]{kruse2025}
with \cite[Proposition 2.7, p.~958]{kruse2025} replaced by \prettyref{prop:isom_strongly_unique_equivalent}. 
We only note that the topological isomorphism $\lambda\colon Y\to Z$ for another Banach predual $(Z,\varphi)$ of 
$\F$ constructed in the proof of the implication (b)$\Rightarrow$(a) of \cite[Proposition 3.8, p.~961]{kruse2025}, 
if the family $(\delta(x))_{x\in\Omega}$ is linearly independent, is already an isometric isomorphism 
if $T$ and $\varphi$ are isometric isomorphisms by \prettyref{prop:equiv_implies_isom_equiv} since it fulfils $\lambda^{t}=T\circ\varphi^{-1}$. 
\end{proof}

Regarding assertion (b) of \prettyref{prop:strongly_unique_isom_predual}, we note that $T(\cdot)(\delta(x))=\delta_{x}$ 
where $\delta_{x}$ is the point evaluation functional given by $\delta_{x}\colon\F\to\K$, $\delta_{x}(f)\coloneqq f(x)$, 
for $x\in\Omega$. Further, we recall the following definition. For a dual Banach space $X$ with Banach predual $(Y,\varphi)$ 
we define the system of seminorms 
\[
p_{N}(x)\coloneqq\sup_{y\in N}|\varphi(x)(y)|,\quad x\in X,
\]
for finite sets $N\subset Y$, which induces a locally convex Hausdorff topology on $X$ 
w.r.t.~the dual paring $\langle X,Y,\varphi\rangle$ and we denote this topology by 
$\sigma_{\varphi}(X,Y)$. With this definition at hand we make the following useful observation, 
which follows from the Banach--Dieudonn\'e theorem and \cite[Chap.~IV, \S1, 1.2, p.~124]{schaefer1971}.

\begin{rem}\label{rem:weakly_continuous}
Let $(X,\|\cdot\|)$ be an isometric dual Banach space over the scalar field $\K$ with isometric Banach predual $(Z,\varphi)$ 
and $x'\colon X\to \K$ a linear map. Then the following assertions are equivalent. 
\begin{enumerate}
\item[(a)] $\ker (x')\cap B_{\|\cdot\|}$ is $\sigma_{\varphi}(X,Z)$-closed.
\item[(b)] $x'$ is $\sigma_{\varphi}(X,Z)$-continuous.
\item[(c)] There is a (unique) $z\in Z$ such that $x'=\varphi(\cdot)(z)$.
\end{enumerate}
\end{rem}

Thus assertion (b) of \prettyref{prop:strongly_unique_isom_predual} is fulfilled if and only for every isometric Banach predual 
$(Z,\varphi)$ of $\F$ and every $x\in\Omega$ it holds that $T(\cdot)(\delta(x))\in(\F,\sigma_{\varphi}(\F,Z))'$. 
In the phrasing of \cite[p.~412]{davidson2011} this means that $T(\cdot)(\delta(x))$ is \emph{universally weak-$\ast$ continuous} 
for every $x\in\Omega$. Analogous to \prettyref{prop:strongly_unique_isom_predual} we have a counterpart of \cite[Corollary 3.9, p.~962]{kruse2025} in the isometric setting as well. 

\begin{cor}\label{cor:unique_isom_predual}
Let $\F$ be a Banach space of $\K$-valued functions on a non-empty set $\Omega$ and $(\delta,Y,T)$ a strong 
isometric Banach linearisation of $\F$. Consider the following assertions.
\begin{enumerate}
\item[(a)] $\F$ has a unique isometric Banach predual. 
\item[(b)] For every isometric Banach predual $(Z,\varphi)$ of $\F$ there is an isometric isomorphism 
$\psi\colon\F\to Z^{\ast}$ such that for every $x\in\Omega$ there is a (unique) $z_{x}\in Z$ with 
$T(\cdot)(\delta(x))=\psi(\cdot)(z_{x})$.
\item[(c)] For every isometric Banach predual $(Z,\varphi)$ of $\F$ there is an isometric isomorphism 
$\psi\colon\F\to Z^{\ast}$ and a (unique) map $\widetilde{\delta}\colon\Omega\to Z$ such that $(\widetilde{\delta},Z,\psi)$ 
is a strong isometric Banach linearisation of $\F$.
\end{enumerate}
Then it holds that (a)$\Rightarrow$(b)$\Leftrightarrow$(c). If the family $(\delta(x))_{x\in\Omega}$ is linearly independent, 
then it holds that (b)$\Rightarrow$(a). 
\end{cor}

\section{Existence of isometric Banach preduals and prebiduals}
\label{sect:isom_existence}

In this section we give necessary and sufficient conditions for the existence of an isometric Banach predual of a normed space. 
We begin with the definition an isometric prebidual whose existence is a priori a stronger condition than the existence of an isometric 
Banach predual. For this definition we need the following notion. For a locally convex Hausdorff space $X$ we denote by $X_{b}'$ 
the dual space $X'$ equipped with the topology of uniform convergence on bounded subsets of $X$. 

\begin{defn}\label{defn:isom_prebidual}
Let $X$ be a normed space. We call $X$ an \emph{isometric bidual Banach space} if there are a locally convex Hausdorff space $Y$ 
such that $Y_{b}'$ is \emph{completely normable} by a norm $\vertiii{\cdot}$, i.e.~$(Y',\vertiii{\cdot})$ is a Banach space and 
the identity map $\id\colon Y_{b}'\to (Y',\vertiii{\cdot})$ is a topological isomorphism, and there is an isometric isomorphism 
$\varphi\colon X \to (Y',\vertiii{\cdot})^{\ast}$. The tuple $(Y,\varphi)$ is called an \emph{isometric prebidual} of $X$. 
\end{defn}

We note that if $X$ is an isometric bidual Banach space with isometric prebidual $(Y,\varphi)$, 
then $X$ is an isometric dual Banach space with isometric Banach predual $((Y',\vertiii{\cdot}),\varphi)$. 
Next, we make use of the mixed topology, \cite[Section 2.1]{wiweger1961}, and the notion of a Saks space, 
\cite[I.3.2 Definition, p.~27--28]{cooper1978}, to give a sufficient condition for the existence of an isometric prebidual.

\begin{defn}[{\cite[Definition 2.2, p.~3]{kruse_schwenninger2022}}]\label{defn:mixed_top_Saks}
Let $(X,\|\cdot\|)$ be a normed space and $\tau$ a Hausdorff locally convex topology on $X$ such that $\tau\leq\tau_{\|\cdot\|}$. 
Then
\begin{enumerate}
\item the \emph{mixed topology} $\gamma \coloneqq \gamma(\|\cdot\|,\tau)$ is
the finest linear topology on $X$ that coincides with $\tau$ on $\|\cdot\|$-bounded sets and such that 
$\tau\leq\gamma\leq\tau_{\|\cdot\|}$,
\item the triple $(X,\|\cdot\|,\tau)$ is called a \emph{pre-Saks space}. It is called a \emph{Saks space} 
if there exists a directed system of continuous seminorms $\Gamma_{\tau}$ that generates the topology $\tau$ such that   
\begin{equation}\label{eq:saks}
\|x\|=\sup_{q\in\Gamma_{\tau}} q(x), \quad x\in X.
\end{equation}
\end{enumerate}
\end{defn}

In comparison to \cite[Definition 2.2, p.~3]{kruse_schwenninger2022} we dropped the assumption that the 
space $(X,\|\cdot\|)$ should be complete and added the notion of a pre-Saks space in \prettyref{defn:mixed_top_Saks}.
However, we will see in \prettyref{rem:semireflexive_pre_Saks_already_Saks} (a) that $(X,\|\cdot\|)$ is complete if 
$(X,\|\cdot\|,\tau)$ is a Saks space and $(X,\gamma)$ semi-reflexive. 
The mixed topology $\gamma$ is actually Hausdorff locally convex and the definition given above is equivalent to the one 
introduced by Wiweger \cite[Section 2.1]{wiweger1961} due to \cite[Lemmas 2.2.1, 2.2.2, p.~51]{wiweger1961}. 
In the case that $(X,\|\cdot\|,\tau)$ is a Saks space, our definition of the mixed topology also coincides with the one of 
Cooper \cite[I.1.4 Definition, p.~6]{cooper1978} due to \cite[I.1.5 Proposition, p.~7]{cooper1978} and 
\cite[I.3.1 Lemma, p.~27]{cooper1978}.

\begin{defn}
Let $(X,\|\cdot\|,\tau)$ be a Saks space. 
\begin{enumerate}
\item We call $(X,\|\cdot\|,\tau)$ \emph{complete} if $(X,\gamma)$ is complete.
\item We call $(X,\|\cdot\|,\tau)$ \emph{semi-reflexive} if $(X,\gamma)$ is semi-reflexive.
\item We call $(X,\|\cdot\|,\tau)$ \emph{semi-Montel} if $(X,\gamma)$ is a semi-Montel space.
\end{enumerate}
\end{defn}

The notions of completeness and semi-reflexivity of Saks spaces were already introduced in 
\cite[I.3.2 Definition, p.~27--28]{cooper1978} and \cite[Definition 2.2]{kruse_seifert2022b}. 
The preceding notions may be characterised by topological properties of $B_{\|\cdot\|}=\{x\in X\;|\; \|x\|\leq 1\}$ w.r.t.~$\tau$.

\begin{rem}\label{rem:pre_Saks_unit_ball_char}
Let $(X,\|\cdot\|,\tau)$ be a pre-Saks space. 
\begin{enumerate}
\item $(X,\|\cdot\|,\tau)$ is a Saks space if and only if $B_{\|\cdot\|}$ is $\tau$-closed 
by \cite[I.3.1 Lemma, p.~27]{cooper1978}. 
\item $(X,\|\cdot\|,\tau)$ is a complete Saks space if and only if $B_{\|\cdot\|}$ is $\tau$-complete 
by \cite[I.1.14 Proposition, p.~11]{cooper1978} and part (a) combined with the observation a $\tau$-complete set is already 
$\tau$-closed.
\item $(X,\|\cdot\|,\tau)$ is a semi-reflexive Saks space if and only if $B_{\|\cdot\|}$ is $\sigma(X,(X,\tau)')$-compact 
by \cite[I.1.21 Corollary, p.~16]{cooper1978} and part (a) combined with the observation that a $\sigma(X,(X,\tau)')$-closed set 
is $\tau$-closed as $\sigma(X,(X,\tau)')\leq\tau$.
\item $(X,\|\cdot\|,\tau)$ is a semi-Montel Saks space if and only if $B_{\|\cdot\|}$ is $\tau$-compact 
by \cite[I.1.13 Proposition, p.~11]{cooper1978}. 
\end{enumerate}
\end{rem}

We note that \prettyref{rem:pre_Saks_unit_ball_char} (a) implies that if $(X,\|\cdot\|,\tau)$ is a Saks space, 
then $(X,\|\cdot\|)$ fulfils condition $\operatorname{(BBCl)}$ of \cite[Definition 3.1 (a), p.~1596--1597]{kruse2024} for $\tau$. 
Further, \prettyref{rem:pre_Saks_unit_ball_char} (d) implies that if $(X,\|\cdot\|,\tau)$ is a semi-Montel Saks space, 
then $(X,\|\cdot\|)$ fulfils condition $\operatorname{(BBC)}$ of \cite[Definition 3.1 (b), p.~1596--1597]{kruse2024} for $\tau$. 
A trivial example of a semi-reflexive Saks space which is not semi-Montel is given by $(X,\|\cdot\|,\tau_{\|\cdot\|})$ for 
a reflexive infinite-dimensional Banach space $(X,\|\cdot\|)$. 
An example where $\tau$ does not coincide with the $\|\cdot\|$-topology is the following one. We recall that a Banach space 
$(X,\|\cdot\|)$ has the \emph{Schur property} if all $\sigma(X,X^{\ast})$-convergent sequences are $\|\cdot\|$-convergent 
(see \cite[p.~253]{fabian2011}). 

\begin{rem}\label{rem:semireflexive_but_not_semiMontel}
Let $(X,\|\cdot\|)$ be a Banach space without the Schur property, e.g.~the space $(c_{0},\|\cdot\|_{\infty})$ (see \cite[p.~252]{fabian2011}), 
and denote by $\mu(X^{\ast},X)$ the Mackey topology on $X^{\ast}$. 
Then $(X^{\ast},\|\cdot\|_{X^{\ast}},\mu(X^{\ast},X))$ is a semi-reflexive Saks space by 
\prettyref{rem:pre_Saks_unit_ball_char} (c) and the Banach--Alaoglu theorem, and is not semi-Montel 
by \prettyref{rem:pre_Saks_unit_ball_char} (d) and \cite[Proposition 3.1, p.~275]{schluechtermann1991} 
since $(X,\|\cdot\|)$ does not have the Schur property. We also observe that 
$\gamma(\|\cdot\|_{X^{\ast}},\mu(X^{\ast},X))=\mu(X^{\ast},X)$ by the second example in \cite[p.~593]{conradie2006}.
\end{rem}

\begin{rem}\fakephantomsection\label{rem:semireflexive_pre_Saks_already_Saks}
\begin{enumerate}
\item[(a)] If $(X,\|\cdot\|,\tau)$ is semi-reflexive Saks space, then $(X,\|\cdot\|)$ is complete. Indeed, by \cite[Chap.~IV, 5.5, Corollary 1, p.~144]{schaefer1971} the semi-reflexivity implies that $(X,\gamma)$ is quasi-complete. 
Therefore $(X,\|\cdot\|)$ is complete since $\gamma\leq\tau_{\|\cdot\|}$ and completeness of a normed space is equivalent 
to quasi-completeness. 
\item[(b)] If $(X,\|\cdot\|,\tau)$ is a semi-Montel Saks space, then it is complete. Indeed, this follows 
from \prettyref{rem:pre_Saks_unit_ball_char} (b) and (d).
\item[(c)] If $(X,\|\cdot\|)$ is a normed space and $\tau$ a locally convex Hausdorff topology on $X$ such that $B_{\|\cdot\|}$ 
is $\tau$-compact, then $\tau\leq\tau_{_{\|\cdot\|}}$ and $(X,\|\cdot\|,\tau)$ is a semi-Montel Saks space. 
Indeed, the $\tau$-compactness of $B_{\|\cdot\|}$ implies that the identity map $\id\colon (X,\|\cdot\|)\to (X,\tau)$ is 
bounded, i.e.~it maps $\|\cdot\|$-bounded sets to $\tau$-bounded sets. 
Since the normed space $(X,\|\cdot\|)$ is bornological, this implies that $\id$ is continuous by 
\cite[Proposition 24.13, p.~283]{meisevogt1997} and so $\tau\leq\tau_{\|\cdot\|}$. 
Thus $(X,\|\cdot\|,\tau)$ is a pre-Saks space, yielding the statement by \prettyref{rem:pre_Saks_unit_ball_char} (d).
\item[(d)] If $(X,\|\cdot\|,\tau)$ is a semi-Montel Saks space and $\tau_{0}\leq\tau$ a locally convex Hausdorff topology on $X$, 
then $(X,\|\cdot\|,\tau_{0})$ is a semi-Montel Saks space and $\gamma(\|\cdot\|,\tau)=\gamma(\|\cdot\|,\tau_{0})$. 
Indeed, by the remarks above \cite[Chap.~3, \S9, Proposition 2, p.~231]{horvath1966} we have that $\tau_{0}$ coincides with $\tau$ 
on $B_{\|\cdot\|}$. Thus $(X,\|\cdot\|,\tau_{0})$ is a semi-Montel Saks space by \prettyref{rem:pre_Saks_unit_ball_char} (d) and 
$\gamma(\|\cdot\|,\tau)=\gamma(\|\cdot\|,\tau_{0})$ by \cite[I.1.6 Corollary (ii), p.~6]{cooper1978}.
\item[(e)] Let $(X,\|\cdot\|)$ be an isometric dual Banach space with isometric Banach predual $(Y,\varphi)$. 
Then $B_{\|\cdot\|}(w)\coloneqq\{x\in X\;|\;\|x-w\|\leq 1\}$ is $\sigma_{\varphi}(X,Y)$-compact for any $w\in X$ 
and $(X,\|\cdot\|,\sigma_{\varphi}(X,Y))$ is a semi-Montel Saks space. Indeed, due to the Banach--Alagolu theorem 
$B_{\|\cdot\|_{Y^{\ast}}}(\varphi(w))$ is $\sigma(Y^{\ast},Y)$-compact. As $\varphi$ is an isometric isomorphism, we have 
$\varphi(B_{\|\cdot\|}(w))=B_{\|\cdot\|_{Y^{\ast}}}(\varphi(w))$ and thus $B_{\|\cdot\|}(w)$ is $\sigma_{\varphi}(X,Y)$-compact, implying the statement by \prettyref{rem:pre_Saks_unit_ball_char} (d).
\end{enumerate}
\end{rem}

The next observation links the existence of an isometric prebidual to $\gamma$.

\begin{prop}\label{prop:dixmier_ng_mixed}
Let $(X,\|\cdot\|,\tau)$ be a Saks space, $X_{\gamma}'\coloneqq (X,\gamma)'$ and $\|\cdot\|_{X_{\gamma}'}$ 
denote the restriction of the dual norm $\|\cdot\|_{X^{\ast}}$ to $X_{\gamma}'$.
\begin{enumerate}
\item Then $X_{\gamma}'$ is a $\|\cdot\|_{X^{\ast}}$-closed subspace of $X^{\ast}$ and 
$(X,\gamma)_{b}'$ is completely norm\-able by $\|\cdot\|_{X_{\gamma}'}$.\footnote{Part (a) shows that \cite[Exercise 19.4, p.~201]{treves2006} is wrong because $(X,\gamma)_{b}'$ is normable although $(X,\gamma)$ is in general not normable by \cite[I.1.15 Proposition, p.~12]{cooper1978}. However, $(X,\gamma)$ is a gDF-space by \cite[I.1.27 Remark, p.~19--20]{cooper1978} and 
thus quasi-normable by \cite[12.4.7 Theorem, p.~260]{jarchow1981}.}
\item If $(X,\|\cdot\|,\tau)$ is semi-reflexive, then the evaluation map 
\[
\mathcal{I}\colon(X,\|\cdot\|)\to (X_{\gamma}',\|\cdot\|_{X_{\gamma}'})^{\ast},\; x\longmapsto [x'\mapsto x'(x)], 
\]
is an isometric isomorphism. In particular, $(X,\|\cdot\|)$ is an isometric bidual Banach space with isometric prebidual 
$((X,\gamma),\mathcal{I})$ and isometric Banach predual $((X_{\gamma}',\|\cdot\|_{X_{\gamma}'}),\mathcal{I})$. 
\end{enumerate}
\end{prop}
\begin{proof}
Part (a) follows from \cite[I.1.18 Proposition, p.~15]{cooper1978}. The bijectivity of $\mathcal{I}$ in part (b) 
is a consequence of part (a) and the semi-reflexivity. That $\mathcal{I}$ is an isometry follows from  
\[
  \|\mathcal{I}(x)\|_{(X_{\gamma}',\|\cdot\|_{X_{\gamma}'})^{\ast}}
=\sup_{\substack{x'\in X_{\gamma}' \\ \|x'\|_{X^{\ast}}\leq 1}}|x'(x)|
=\|x\|
\]
for all $x\in X$ where the second equation is due to \cite[I.3.1 Lemma, p.~27]{cooper1978} and 
\cite[Lemma 5.5 (a), p.~2680--2681]{kruse_meichnser_seifert2018}. 
\end{proof}

Comparing \prettyref{prop:dixmier_ng_mixed} (b) with the Dixmier--Ng theorem \cite[Theorem 1, p.~279]{ng1971}, we first note that
\[
X_{\gamma}'=\{x'\colon X\to \K\;|\; x'\;\text{linear on $X$ and $\tau$-continuous on }B_{\|\cdot\|}\}\eqqcolon V
\]
by \cite[I.1.7 Corollary, p.~8]{cooper1978} where $\K=\R$ or $\C$ is the scalar field of $X$. 
The (proof of the) Dixmier--Ng theorem states that the map $\mathcal{I}\colon(X,\|\cdot\|)\to 
(V,{\|\cdot\|_{X^{\ast}}}_{\mid V})^{\ast}$ is an isometric isomorphism if $B_{\|\cdot\|}$ is $\tau$-compact for 
some locally convex Hausdorff topology $\tau$ on $X$. The latter condition implies that $(X,\|\cdot\|,\tau)$ is a semi-Montel 
Saks space by \prettyref{rem:semireflexive_pre_Saks_already_Saks} (c). 
So, a priori \prettyref{prop:dixmier_ng_mixed} (b) is more general than the Dixmier--Ng theorem. However, we will see in 
the next proposition that they are actually equivalent.

\begin{cor}\label{cor:existence_isom_banach_predual}
Let $(X,\|\cdot\|)$ be a normed space. 
Then the following assertions are equivalent.
\begin{enumerate}
\item[(a)] $(X,\|\cdot\|)$ has an isometric Banach predual.
\item[(b)] $(X,\|\cdot\|)$ has an isometric complete semi-Montel prebidual $(Y,\varphi)$.
\item[(c)] $(X,\|\cdot\|)$ has an isometric prebidual $(Y,\varphi)$.
\item[(d)] $(X,\|\cdot\|,\tau)$ is a semi-reflexive Saks space for some locally convex Hausdorff topology $\tau\leq\tau_{\|\cdot\|}$.
\item[(e)] $(X,\|\cdot\|,\tau)$ is a semi-Montel Saks space for some locally convex Hausdorff topology $\tau\leq\tau_{\|\cdot\|}$.
\item[(f)] There is a closed linear subspace $V\subset X^{\ast}$ such that $B_{\|\cdot\|}$ is $\sigma(X,V)$-compact. 
\item[(g)] $B_{\|\cdot\|}$ is $\tau$-compact for some locally convex Hausdorff topology $\tau$ on $X$.
\end{enumerate}
\end{cor}
\begin{proof}
(c)$\Rightarrow$(a) Since $(X,\|\cdot\|)$ has an isometric prebidual $(Y,\varphi)$, the strong dual $Y_{b}'$ is 
completely normable by a norm $\vertiii{\cdot}$ such that $\varphi\colon (X,\|\cdot\|) \to (Y',\vertiii{\cdot})^{\ast}$ is an 
isometric isomorphism. Hence $((Y',\vertiii{\cdot}),\varphi)$ is an isometric Banach predual of $(X,\|\cdot\|)$.

(a)$\Rightarrow$(f) Since $(X,\|\cdot\|)$ has an isometric Banach predual, there are a Banach space $Y$ and 
an isometric isomorphism $\varphi\colon (X,\|\cdot\|)\to Y^{\ast}$. 
Due to \prettyref{rem:semireflexive_pre_Saks_already_Saks} (e) $B_{\|\cdot\|}$ is $\sigma_{\varphi}(X,Y)$-compact. 
Using that $V\coloneqq\Phi_{\varphi}(Y)$ is closed in $X^{\ast}$ by \prettyref{rem:isom_predual_into_dual} 
and $\sigma_{\varphi}(X,Y)=\sigma(X,V)$, we obtain assertion (f).  

(f)$\Rightarrow$(g) Due to \cite[Th\'{e}or\`{e}me 14, p.~1066]{dixmier1948} and \cite[D\'{e}finition 3, p.~1065]{dixmier1948} 
the locally convex topology $\sigma(X,V)$ is Hausdorff,
which implies (g) with $\tau\coloneqq \sigma(X,V)$.

(g)$\Rightarrow$(e) This implication follows from \prettyref{rem:semireflexive_pre_Saks_already_Saks} (c).

(e)$\Rightarrow$(d) This implication follows from the observation that every semi-Montel space is semi-reflexive. 

(d)$\Rightarrow$(c) This implication follows from \prettyref{prop:dixmier_ng_mixed} (b) with 
$Y\coloneqq (X,\gamma)$ and $\varphi\coloneqq\mathcal{I}$. 

(e)$\Rightarrow$(b) This implication follows from \prettyref{prop:dixmier_ng_mixed} (b) with 
$Y\coloneqq (X,\gamma)$ and $\varphi\coloneqq\mathcal{I}$ since $(X,\gamma)$ is a semi-Montel space, 
thus semi-reflexive, and complete by \prettyref{rem:semireflexive_pre_Saks_already_Saks} (b). 

(b)$\Rightarrow$(c) This implication is obvious.
\end{proof}

The equivalence (a)$\Leftrightarrow$(f) of \prettyref{cor:existence_isom_banach_predual} is already given in 
\cite[Th\'{e}or\`{e}me 19, p.~1069]{dixmier1948} of Dixmier\footnote{The condition that $V\subset X^{\ast}$ should be closed is missing in \cite[Th\'{e}or\`{e}me 19, p.~1069]{dixmier1948}. However, due to \cite[Th\'{e}or\`{e}me 16, p.~1068]{dixmier1948}, \cite[l.~10, p.~1069]{dixmier1948} and \cite[Th\'{e}or\`{e}me 14, p.~1066]{dixmier1948} it should be added.}, which inspired the Dixmier--Ng theorem \cite[Theorem 1, p.~279]{ng1971}, i.e.~the implication (g)$\Rightarrow$(a). 

\begin{exa}\label{ex:subspace_cont_mixed}
(i) For a discrete space $\Omega$ and a function $v\colon \Omega\to (0,\infty)$ we set
\[
\ell v(\Omega)\coloneqq\{f\colon\Omega\to \K\;|\;\|f\|_{v}\coloneqq\sup_{x\in\Omega}|f(x)|v(x)<\infty\}
\]
as well as $\ell^{\infty}(\Omega)\coloneqq\ell w(\Omega)$ and $\|f\|_{\infty}\coloneqq \|f\|_{w}$ with 
$w(x)\coloneqq 1$ for all $x\in\Omega$. 
We note that the multiplication operator
\[
M_{v}\colon\ell v(\Omega)\to\ell^{\infty}(\Omega),\;M_{v}(f)\coloneqq fv,
\]
is an isometric isomorphism w.r.t~to the norms $\|\cdot\|_{v}$ and $\|\cdot\|_{\infty}$, and a topological isomorphism if both spaces 
are equipped with the compact-open topology $\tau_{\operatorname{co}}$. By \cite[II.1.24 Remark 4), p.~88--89]{cooper1978} 
$B_{\|\cdot\|_{\infty}}$ is $\tau_{\operatorname{co}}$-compact. Hence $B_{\|\cdot\|_{v}}$ is $\tau_{\operatorname{co}}$-compact due 
to the isomorphism $M_{v}$. Thus $(\ell v(\Omega),\|\cdot\|_{v},\tau_{\operatorname{co}})$ is a complete semi-Montel Saks space 
by \prettyref{rem:pre_Saks_unit_ball_char} (d) and \prettyref{rem:semireflexive_pre_Saks_already_Saks} (b). 

(ii) For an open set $\Omega\subset\R^{d}$ we define the kernel
\[
\mathcal{C}_{P}(\Omega)\coloneqq\{f\in\mathcal{C}^{\infty}(\Omega)\;|\;f\in\ker P(\partial)\}
\]
of a hypoelliptic linear partial differential operator $P(\partial)\colon\mathcal{C}^{\infty}(\Omega)\to\mathcal{C}^{\infty}(\Omega)$ 
where $\mathcal{C}^{\infty}(\Omega)$ is the space of infinitely continuously partially differentiable $\K$-valued functions 
on $\Omega$. For a continuous function $v\colon\Omega\to(0,\infty)$ we define the weighted kernel
\[
  \mathcal{C}_{P}v(\Omega)
\coloneqq \{f\in\mathcal{C}_{P}(\Omega)\;|\;\|f\|_{v}\coloneqq\sup_{x\in\Omega}|f(x)|v(x)<\infty\}.
\]
$B_{\|\cdot\|_{v}}$ is compact in the Montel space $(\mathcal{C}_{P}(\Omega),\tau_{\operatorname{co}})$ 
by the proof of \cite[5.2.30 Corollary, p.~87]{kruse2023} and so in $(\mathcal{C}_{P}v(\Omega),\tau_{\operatorname{co}})$ 
as well. Thus $(\mathcal{C}_{P}v(\Omega),\|\cdot\|_{v},\tau_{\operatorname{co}})$ is a complete semi-Montel Saks space 
by \prettyref{rem:pre_Saks_unit_ball_char} (d) and \prettyref{rem:semireflexive_pre_Saks_already_Saks} (b). 

(iii) For an open subset $\Omega$ of a complex locally convex Hausdorff $k$-space 
let $\mathcal{H}(\Omega)$ be the space of holomorphic functions $f\colon\Omega\to \C$, i.e.~the space of G\^{a}teaux-holomorphic 
and continuous functions $f\colon\Omega\to \C$ (see \cite[Definition 3.6, p.~152]{dineen1999}),
and for a continuous function $v\colon\Omega\to (0,\infty)$ we set 
\[
\mathcal{H}v(\Omega)\coloneqq \{f\in \mathcal{H}(\Omega)\;|\;\|f\|_{v}\coloneqq\sup_{x\in\Omega}|f(x)|v(x)<\infty\}.
\]
It is easily seen that $B_{\|\cdot\|_{v}}$ is closed in $(\mathcal{H}(\Omega),\tau_{\operatorname{co}})$.
Further, $(\mathcal{H}(\Omega),\tau_{\operatorname{co}})$ is a semi-Montel space 
by \cite[Proposition 3.37, p.~130]{dineen1981} since $\Omega$ is an open subset of a locally convex Hausdorff $k$-space. 
This implies that $B_{\|\cdot\|_{v}}$ is compact in $(\mathcal{H}(\Omega),\tau_{\operatorname{co}})$ 
and so in $(\mathcal{H}v(\Omega),\tau_{\operatorname{co}})$ as well. 
Thus $(\mathcal{H}v(\Omega),\|\cdot\|_{v},\tau_{\operatorname{co}})$ is a complete semi-Montel Saks space 
by \prettyref{rem:pre_Saks_unit_ball_char} (d) and \prettyref{rem:semireflexive_pre_Saks_already_Saks} (b). 

(iv) For a continuous function $v\colon\D\to(0,\infty)$ with $\D\coloneqq\{z\in\C\;|\;|z|<1\}$ we define the Bloch type space 
\[
\mathcal{B}v(\D)\coloneqq\{f\in\mathcal{H}(\D)\;|\;\|f\|_{\mathcal{B}v}\coloneqq |f(0)|+\sup_{z\in\D}|f'(z)|v(z)<\infty\}.
\]
By the Weierstra{\ss} theorem $B_{\|\cdot\|_{\mathcal{B}v}}$ is closed in the Montel space 
$(\mathcal{H}(\D),\tau_{\operatorname{co}})$.
Hence the set $B_{\|\cdot\|_{\mathcal{B}v}}$ is also compact in $(\mathcal{B}v(\D), \tau_{\operatorname{co}})$ (cf.~\cite[p.~4]{eklund2017} for radial, non-increasing $v$, and \cite[Corollary 3.8, p.~9--10]{kruse2019_3} for general $v$). 
Thus $(\mathcal{B}v(\D),\|\cdot\|_{\mathcal{B}v},\tau_{\operatorname{co}})$ is a complete semi-Montel Saks space 
by \prettyref{rem:pre_Saks_unit_ball_char} (d) and \prettyref{rem:semireflexive_pre_Saks_already_Saks} (b). 

(v) For a metric space $(\Omega,\d)$ with a base point denoted by $\mathbf{0}$, 
i.e.~a \emph{pointed metric space} in the sense of \cite[p.~1]{weaver2018}, 
we define the space of $\K$-valued Lipschitz continuous functions on $(\Omega,\d)$ that vanish at $\mathbf{0}$ by 
\[
\mathrm{Lip}_{0}(\Omega)
\coloneqq\{f\colon\Omega\to \K\;|\;f(\mathbf{0})=0\;\text{and}\;\|f\|_{\mathrm{Lip}}\coloneqq\sup_{\substack{x,y\in\Omega\\x\neq y}}
\frac{|f(x)-f(y)|}{\d(x,y)}<\infty\}.
\] 
For all $f\in B_{\|\cdot\|_{\mathrm{Lip}}}$ we have
\[
|f(x)-f(y)|\leq \d(x,y),\quad x,y\in\Omega,
\]
which implies 
\[
|f(x)|=|f(x)-f(\mathbf{0})|\leq \d(x,\mathbf{0}),\quad x\in\Omega.
\]
It follows that $B_{\|\cdot\|_{\mathrm{Lip}}}$ is (uniformly) equicontinuous and 
$\{f(x)\;|\;f\in B_{\|\cdot\|_{\mathrm{Lip}}}\}$ is bounded in $\K$ for all $x\in\Omega$. 
Ascoli's theorem (see e.g.~\cite[Theorem 47.1, p.~290]{munkres2000}) implies that $B_{\|\cdot\|_{\mathrm{Lip}}}$
is compact in $(\mathcal{C}(\Omega),\tau_{\operatorname{co}})$ and so in $(\mathrm{Lip}_{0}(\Omega), \tau_{\operatorname{co}})$ 
as well. Thus $(\mathrm{Lip}_{0}(\Omega),\|\cdot\|_{\mathrm{Lip}},\tau_{\operatorname{co}})$ is a complete semi-Montel Saks space 
by \prettyref{rem:pre_Saks_unit_ball_char} (d) and \prettyref{rem:semireflexive_pre_Saks_already_Saks} (b). 

(vi) For $k\in\N_{0}$ and an open bounded set $\Omega\subset\R^{d}$ let $\mathcal{C}^{k}(\Omega)$ denote the space of $k$-times 
continuously partially differentiable $\K$-valued functions on $\Omega$. 
We define the space of $k$-times continuously partially differentiable $\K$-valued functions on $\Omega$ whose partial derivatives 
up to order $k$ are continuously extendable to the boundary of $\Omega$ by 
\[
 \mathcal{C}^{k}(\overline{\Omega})\coloneqq\{f\in\mathcal{C}^{k}(\Omega)\;|\;\partial^{\beta}f\;
 \text{cont.\ extendable on}\;\overline{\Omega}\;\text{for all}\;\beta\in\N^{d}_{0},\,|\beta|\leq k\}
\]
which we equip with the norm given by 
\[
 |f|_{\mathcal{C}^{k}}\coloneqq \sup_{\substack{x\in \Omega\\ \beta\in\N^{d}_{0}, |\beta|\leq k}}
 |\partial^{\beta}f(x)|, \quad f\in\mathcal{C}^{k}(\overline{\Omega}).
\]
The space of functions in $\mathcal{C}^{k}(\overline{\Omega})$ such 
that all its $k$-th partial derivatives are $\alpha$-H\"older continuous with $0<\alpha\leq 1$ is given by  
\[
\mathcal{C}^{k,\alpha}(\overline{\Omega})\coloneqq 
\bigl\{f\in\mathcal{C}^{k}(\overline{\Omega})\;|\;\|f\|_{\mathcal{C}^{k,\alpha}}<\infty\bigr\}
\]
where
\[
\|f\|_{\mathcal{C}^{k,\alpha}}\coloneqq |f|_{\mathcal{C}^{k}}
+\sup_{\beta\in\N^{d}_{0}, |\beta|=k}\sup_{\substack{x,y\in\Omega\\x\neq y}}
\frac{|\partial^{\beta}f(x)-\partial^{\beta}f(y)|}{|x-y|^{\alpha}}.
\]
If $k\geq 1$, we assume additionally that $\Omega$ has Lipschitz boundary. The set $B_{\|\cdot\|_{\mathcal{C}^{k,\alpha}}}$ 
is relatively $|\cdot|_{\mathcal{C}^{k}}$-compact in $\mathcal{C}^{k}(\overline{\Omega})$ 
by \cite[8.6 Einbettungssatz in H\"older-R\"aumen, p.~338]{alt2012}, and it is easily seen that is also 
$|\cdot|_{\mathcal{C}^{k}}$-closed by a pointwise argument. 
Thus $(\mathcal{C}^{k,\alpha}(\overline{\Omega}),\|\cdot\|_{\mathcal{C}^{k,\alpha}},\tau_{|\cdot|_{\mathcal{C}^{k}}})$ 
is a complete semi-Montel Saks space by \prettyref{rem:pre_Saks_unit_ball_char} (d) 
and \prettyref{rem:semireflexive_pre_Saks_already_Saks} (b). 
\end{exa}

\prettyref{ex:subspace_cont_mixed} (iii) is already contained in \cite[Propositions 3.2, p.~78]{bierstedt_summers1993} 
if $\Omega$ is an open subset of $\C^{d}$. \prettyref{ex:subspace_cont_mixed} (v) is already contained 
in \cite[Theorem 2.1 (7), p.~642]{vargas2018} but the proof of the $\tau_{\operatorname{co}}$-compactness of 
$B_{\|\cdot\|_{\mathrm{Lip}_{0}}}$ is only sketched (see \cite[p.~641]{vargas2018}), which is why we included it here. 

\begin{exa}\label{ex:cont_lin_op_mixed}
Let $(X,\|\cdot\|_{X})$ and $(Y,\|\cdot\|_{Y})$ be Banach spaces and $L(X,Y)$ denote the space of continuous linear operators 
from $X$ to $Y$. If $X=Y$, we just write $L(X)\coloneqq L(X,X)$. The operator norm on $L(X,Y)$ is given by 
\[
\|T\|_{L(X,Y)}\coloneqq\sup_{x\in B_{\|\cdot\|_{X}}}\|Tx\|_{Y},\quad T\in L(X,Y).
\]
The weak operator topology $\tau_{\operatorname{wot}}$ on $L(X,Y)$ is 
induced by the directed system of seminorms given by 
\[
p_{N,M}(T)\coloneqq\sup_{x\in N, y'\in M}|y'(Tx)|,\quad T\in L(X,Y),
\]
for finite $N\subset X$ and finite $M\subset Y'$. If $Y$ is reflexive, then $B_{\|\cdot\|_{L(X,Y)}}$ is 
$\tau_{\operatorname{wot}}$-compact by \cite[Theorem 2.19, p.~1689]{choi2008} 
(cf.~\cite[Example 3.11 (d), p.~10]{kruse_schwenninger2022}).
Thus $(L(X,Y),\|\cdot\|_{L(X,Y)},\tau_{\operatorname{wot}})$ is a complete semi-Montel Saks space 
by \prettyref{rem:pre_Saks_unit_ball_char} (d) and \prettyref{rem:semireflexive_pre_Saks_already_Saks} (b) 
if $Y$ is reflexive. 
\end{exa}

The special case of \prettyref{ex:cont_lin_op_mixed} that $X=Y$ is a Hilbert space is already contained in 
\cite[IV.2.3 Proposition 5), p.~204--205]{cooper1978}.

\section{Isometric linearisation}
\label{sect:isom_linearisation}

Let $(\F,\|\cdot\|,\tau)$ be a semi-reflexive Saks space of $\K$-valued functions on a non-empty set $\Omega$. 
Suppose that $\delta_{x}\in \F_{\gamma}'$ for all $x\in\Omega$. 
Then $\mathcal{I}\colon(\F,\|\cdot\|)\to (\F_{\gamma}',\|\cdot\|_{\F_{\gamma}'})^{\ast}$ is an isometric isomorphism 
by \prettyref{prop:dixmier_ng_mixed} (b) and we have $\mathcal{I}(f)(\delta_{x})=\delta_{x}(f)=f(x)$ for all $f\in\F$ and 
$x\in\Omega$. Thus we have found a sufficient condition for the existence of a strong isometric Banach linearisation of $\F$.

\begin{cor}\label{cor:siB_linearisation}
Let $(\F,\|\cdot\|,\tau)$ be a semi-reflexive Saks space of $\K$-valued functions on a non-empty set $\Omega$. 
If $\Delta(x)\coloneqq\delta_{x}\in \F_{\gamma}'$ for all $x\in\Omega$, then 
$(\Delta,\F_{\gamma}',\mathcal{I})$ is a strong isometric Banach linearisation of $\F$.
\end{cor}

\begin{rem}\label{rem:point_eval_in_predual}
Let $(\F,\|\cdot\|,\tau)$ be a Saks space of $\K$-valued functions on a non-empty set $\Omega$. 
If $\tau_{\operatorname{p}}\leq\tau$, then $\delta_{x}\in\F_{\gamma}'$ for all $x\in\Omega$ by 
\cite[Propositions 3.16, 3.17, p.~1602]{kruse2024} and \cite[Remark 4.3, p.~1609]{kruse2024}.
\end{rem}

Our next result tells us that our sufficient condition is also necessary. 
For its proof we only need to modify the proofs of \cite[Theorem 4.5, p.~1609]{kruse2024} and 
\cite[Corollary 4.8, p.~1611]{kruse2024}.

\begin{thm}\label{thm:existence_siB_linearisation}
Let $(\F,\|\cdot\|)$ be a Banach space of $\K$-valued functions on a non-empty set $\Omega$.
Then the following assertions are equivalent.
\begin{enumerate}
\item[(a)] $(\F,\|\cdot\|)$ admits a strong isometric Banach linearisation.
\item[(b)] $(\F,\|\cdot\|)$ has an isometric complete semi-Montel prebidual $(Y,\varphi)$ and for every $x\in\Omega$ 
there is a unique $y_{x}'\in Y'$ such that $\delta_{x}=\varphi(\cdot)(y_{x}')$.
\item[(c)] $(\F,\|\cdot\|)$ has an isometric prebidual $(Y,\varphi)$ and for every $x\in\Omega$ 
there is a unique $y_{x}'\in Y'$ such that $\delta_{x}=\varphi(\cdot)(y_{x}')$.
\item[(d)] $(\F,\|\cdot\|,\tau)$ is a semi-reflexive Saks space for some locally convex Hausdorff topology 
$\tau_{\operatorname{p}}\leq\tau\leq\tau_{\|\cdot\|}$.
\item[(e)] $(\F,\|\cdot\|,\tau)$ is a semi-Montel Saks space for some some locally convex Hausdorff topology 
$\tau_{\operatorname{p}}\leq\tau\leq\tau_{\|\cdot\|}$.
\item[(f)] $B_{\|\cdot\|}$ is $\tau$-compact for some some locally convex Hausdorff topology 
$\tau_{\operatorname{p}}\leq\tau$.
\item[(g)] $B_{\|\cdot\|}$ is $\tau$-compact for some locally convex Hausdorff topology $\tau$ on $\F$ 
such that $\delta_{x}\in \F_{\gamma}'$ for all $x\in\Omega$.
\item[(h)] $B_{\|\cdot\|}$ is $\tau_{\operatorname{p}}$-compact.
\end{enumerate}
\end{thm}
\begin{proof}
(c)$\Rightarrow$(a) By the proof of \prettyref{cor:existence_isom_banach_predual} there is a norm $\vertiii{\cdot}$ on $Y'$ 
such that tuple $((Y',\vertiii{\cdot}),\varphi)$ is an isometric Banach predual of $(\F,\|\cdot\|)$. 
We set $\delta\colon\Omega\to Y'$, $\delta(x)\coloneqq y_{x}'$. Then we have 
\[
(\varphi(f)\circ\delta)(x)=\varphi(f)(y_{x}')=\delta_{x}(f)=f(x)
\]
for all $f\in\F$ and $x\in\Omega$. Hence $(\delta,(Y',\vertiii{\cdot}),\varphi)$ is a strong isometric Banach linearisation of $\F$. 

(a)$\Rightarrow$(f) As $(\F,\|\cdot\|)$ admits a strong isometric Banach linearisation, there are a Banach space $Y$, 
a map $\delta\colon\Omega\to Y$ and an isometric isomorphism $T\colon(\F,\|\cdot\|)\to Y^{\ast}$ such that $T(f)\circ\delta=f$ 
for all $f\in\F$. From the proof of \prettyref{cor:existence_isom_banach_predual} we obtain that $B_{\|\cdot\|}$ is 
$\sigma_{T}(\F,Y)$-compact and set $\tau\coloneqq \sigma_{T}(\F,Y)$.
Looking at $T(f)(\delta(x))=f(x)$ for all $f\in\F$ and $x\in\Omega$, we observe that $\tau_{\operatorname{p}}$ 
and $\sigma_{T}(\F,Y_{0})$ coincide 
on $\F$ where $Y_{0}$ denotes the span of $\{\delta(x)\;|\;x\in\Omega\}$ which is dense in $Y$ 
by \cite[Proposition 2.7, p.~1596]{kruse2024}, and $\sigma_{T}(\F,Y_{0})$ is defined 
by the system of seminorms 
\[
p_{N}(f)\coloneqq\sup_{y\in N}|T(f)(y)|,\quad f\in \F,
\]
for finite sets $N\subset Y_{0}$. Hence we have $\tau_{\operatorname{p}}=\sigma_{T}(\F,Y_{0})\leq\sigma_{T}(\F,Y)=\tau$.

(f)$\Rightarrow$(e), (e)$\Rightarrow$(d) These implications follow from the proof of \prettyref{cor:existence_isom_banach_predual}. 

(d)$\Rightarrow$(c) Since $\tau_{\operatorname{p}}\leq\tau$, we have $\delta_{x}\in \F_{\gamma}'$ for all $x\in\Omega$ 
by \prettyref{rem:point_eval_in_predual}. It follows from \prettyref{cor:siB_linearisation} with 
$Y\coloneqq (\F,\gamma)$ and $\varphi\coloneqq\mathcal{I}$ that $(Y,\varphi)$ is an isometric prebidual of $(\F,\|\cdot\|)$ 
and $\delta_{x}=\varphi(\cdot)(y_{x}')$ with $y_{x}'\coloneqq\delta_{x}\in Y'$. 
The uniqueness of $y_{x}'$ follows from \prettyref{rem:isom_predual_into_dual}. 

(e)$\Rightarrow$(b) Using that $(\F,\gamma)$ is a complete semi-Montel space by 
\prettyref{rem:semireflexive_pre_Saks_already_Saks} (b), 
this implication follows analogously to the proof of the implication (d)$\Rightarrow$(c).

(b)$\Rightarrow$(c) This implication is obvious.

(f)$\Rightarrow$(g) This implication follows from \prettyref{rem:point_eval_in_predual}.

(g)$\Rightarrow$(a) Since $B_{\|\cdot\|}$ is $\tau$-compact, we have that $(\F,\|\cdot\|,\tau)$ 
is a semi-Montel Saks space by \prettyref{rem:semireflexive_pre_Saks_already_Saks} (c). Hence assertion (a) follows from \prettyref{cor:siB_linearisation}.

(f)$\Rightarrow$(h) This implication follows from \prettyref{rem:semireflexive_pre_Saks_already_Saks} (d).

(h)$\Rightarrow$(f) This implication is obvious with $\tau\coloneqq \tau_{\operatorname{p}}$.
\end{proof}

\begin{exa}
For all complete semi-Montel Saks spaces $(\F,\|\cdot\|,\tau)$ 
from \prettyref{ex:subspace_cont_mixed} we have $\tau_{\operatorname{p}}\leq\tau$ and thus 
$(\Delta,\F_{\gamma}',\mathcal{I})$ is a strong isometric Banach linearisation of $\F$ in all cases 
by \prettyref{cor:siB_linearisation} and \prettyref{rem:point_eval_in_predual}.
\end{exa}

We note the following generalisation of \cite[Corollary 2.3, p.~642]{vargas2018}, where $\F=\operatorname{Lip}_{0}(\Omega)$, 
and \cite[8.~Corollary, p.~292]{prieto1992}, where $\F=H^{\infty}(\Omega)=\mathcal{H}w(\Omega)$ with $w(z)\coloneqq 1$ for $z\in\Omega$ 
and $\Omega\subset\C$ is an open connected set. 
If $(\F,\|\cdot\|,\tau)$ is a semi-reflexive Saks space such that $\tau_{\operatorname{p}}\leq\tau$, 
then $\gamma(\|\cdot\|,\tau_{\operatorname{p}})$ coincides with the \emph{bounded weak$^{\ast}$ topology} 
$b\sigma_{\mathcal{I}}(\F, \F_{\gamma}')$ of the dual pairing $\langle \F, \F_{\gamma}', \mathcal{I}\rangle$, 
i.e.~the finest locally convex Hausdorff topology on $\F$ which coincides with the topology $\sigma_{\mathcal{I}}(\F, \F_{\gamma}')$ 
on $\|\cdot\|$-bounded sets (see e.g.~\cite[p.~292]{prieto1992}).

\begin{prop}\label{prop:bounded_weak_star}
Let $(\F,\|\cdot\|,\tau)$ be a semi-reflexive Saks space of $\K$-valued functions on a non-empty set $\Omega$ such that $\tau_{\operatorname{p}}\leq\tau$ and $\gamma\coloneqq\gamma(\|\cdot\|,\tau)$. Then it holds that
\[
 \gamma(\|\cdot\|,\tau_{\operatorname{p}})
=\gamma(\|\cdot\|,\sigma_{\mathcal{I}}(\F,\F_{\gamma}'))
=b\sigma_{\mathcal{I}}(\F,\F_{\gamma}')
=\tau_{\operatorname{c},\mathcal{I}}(\F,\F_{\gamma}')
\]
where $\tau_{\operatorname{c},\mathcal{I}}(\F,\F_{\gamma}')$ denotes the locally convex Hausdorff topology on $\F$ that is induced 
by the system of seminorms 
\[
p_{K}(f)\coloneqq\sup_{f'\in K}|\mathcal{I}(f)(f')|=\sup_{f'\in K}|f'(f)|,\quad f\in\F,
\]
for $\|\cdot\|_{\F_{\gamma}'}$-compact sets $K\subset \F_{\gamma}'$. If $(\F,\|\cdot\|,\tau)$ is even semi-Montel, then 
$\gamma=\gamma(\|\cdot\|,\tau_{\operatorname{p}})$.
\end{prop}
\begin{proof}
Since $\delta_{x}\in\F_{\gamma}'$ for all $x\in\Omega$ by \prettyref{rem:point_eval_in_predual}, we get $\tau_{\operatorname{p}}\leq\sigma_{\mathcal{I}}(\F,\F_{\gamma}')$. 
Due to \prettyref{rem:semireflexive_pre_Saks_already_Saks} (e) $(\F,\|\cdot\|,\sigma_{\mathcal{I}}(\F,\F_{\gamma}'))$ 
is a semi-Montel Saks space and $\gamma(\|\cdot\|,\tau_{\operatorname{p}})=\gamma(\|\cdot\|,\sigma_{\mathcal{I}}(\F,\F_{\gamma}'))$ 
by \prettyref{rem:semireflexive_pre_Saks_already_Saks} (d). Further, we have
\[
 \gamma(\|\cdot\|,\sigma_{\mathcal{I}}(\F,\F_{\gamma}'))
=b\sigma_{\mathcal{I}}(\F,\F_{\gamma}')
=\tau_{\operatorname{c},\mathcal{I}}(\F,\F_{\gamma}')
\]
by \cite[Example E), p.~66]{wiweger1961} for the second equation and \prettyref{defn:mixed_top_Saks} (a) for the first equation 
since the mixed topology $\gamma(\|\cdot\|,\sigma_{\mathcal{I}}(\F,\F_{\gamma}'))$ is a locally convex Hausdorff topology. 
If $(\F,\|\cdot\|,\tau)$ is even semi-Montel, then 
$\gamma=\gamma(\|\cdot\|,\tau_{\operatorname{p}})$ by \prettyref{rem:semireflexive_pre_Saks_already_Saks} (d).
\end{proof}

There is also a vector-valued version of \prettyref{cor:siB_linearisation} for which we need a weak vector-valued version 
of the function space $\F$. We recall its definition from \cite[p.~963]{kruse2025}. 
Let $(\F,\|\cdot\|)$ be a normed space of $\K$-valued functions on a non-empty set $\Omega$. 
For a normed space $E$ over the field $\K$ we define the space
\[
\FE_{\sigma}\coloneqq\{f\colon \Omega\to E\;|\;\forall\;e^{\ast}\in E^{\ast}:\;e^{\ast}\circ f\in\F\}.
\]
Further, we set
\[
\|f\|_{\sigma}\coloneqq \sup_{e^{\ast}\in B_{\|\cdot\|_{E^{\ast}}}}\|e^{\ast}\circ f\|,\quad f\in\FE_{\sigma}.
\]
If $(\F,\|\cdot\|)$ is a Banach space such that $\tau_{\operatorname{p}}\leq\tau_{\|\cdot\|}$, then $\|f\|_{\sigma}<\infty$ for all $f\in\FE_{\sigma}$ 
by \cite[Remark 4.1, p.~964]{kruse2025} and thus $(\FE_{\sigma},\|\cdot\|_{\sigma})$ is a normed space.

Now, let $(\F,\|\cdot\|,\tau)$ be a semi-reflexive Saks space of $\K$-valued functions on a non-empty set $\Omega$ such that $\tau_{\operatorname{p}}\leq\tau$ and $E$ a Banach space over the field $\K$. Then the map 
\begin{equation}\label{eq:si_lin_vec_valued}
\chi\colon L(\F_{\gamma}',E)\to\FE_{\sigma},\;\chi(u)\coloneqq u\circ \Delta,
\end{equation}
is an isometric isomorphism w.r.t.~the norms $\|\cdot\|_{L(\F_{\gamma}',E)}$ and $\|\cdot\|_{\sigma}$ 
by \prettyref{cor:siB_linearisation}, \prettyref{rem:point_eval_in_predual}, 
\cite[Remark 4.4, p.~964]{kruse2025} and \cite[Theorem 4.5, p.~965]{kruse2025}. 
We may use this vector-valued result in connection to (weakly) compact composition operators. 
Let $\varphi\colon\Omega\to\Omega$ be such that $f\circ\varphi\in\F$ for all $f\in\F$. 
Then the composition operators $C_{\varphi}\colon\F\to\F$ and $C_{\varphi}\colon\FE_{\sigma}\to\FE_{\sigma}$ 
given by $C_{\varphi}(f)\coloneqq f\circ\varphi$ are well-defined linear maps since $e^{\ast}\circ f\in\F$ for every 
$e^{\ast}\in E^{\ast}$ and $f\in\FE_{\sigma}$. 
We recall that a linear map $A\colon X\to Y$ between two Banach spaces $(X,\|\cdot\|_{X})$ and $(Y\|\cdot\|_{Y})$ 
is called \emph{compact} if $A(B_{\|\cdot\|_{X}})$ is relatively $\|\cdot\|_{Y}$-compact. 
A linear map $A\colon X\to Y$ is called \emph{weakly compact} if $A(B_{\|\cdot\|_{X}})$ is relatively $\sigma(Y,Y^{\ast})$-compact 
(see e.g.~\cite[p.~235]{bonet2001}). The following theorem generalises \cite[Proposition 11, p.~244]{bonet2001} and 
\cite[Theorem 5.3, p.~14]{laitila2006}. Its proof is based on the one of \cite[Proposition 11, p.~244]{bonet2001}.

\begin{thm}\label{thm:compact_operator}
Let $(\F,\|\cdot\|,\tau)$ be a semi-reflexive Saks space of $\K$-valued functions on a non-empty set $\Omega$ such that $\tau_{\operatorname{p}}\leq\tau$, $\varphi\colon\Omega\to\Omega$ such that 
$f\circ\varphi\in\F$ for all $f\in\F$ and $E$ a Banach space over the field $\K$. 
\begin{enumerate}
\item If $C_{\varphi}\in L(\F)$, then $C_{\varphi}\in L(\FE_{\sigma})$. 
\item If $C_{\varphi}\colon\F\to\F$ is compact and $E$ reflexive, 
then $C_{\varphi}\colon\FE_{\sigma}\to\FE_{\sigma}$ is weakly compact. 
\end{enumerate}
\end{thm}
\begin{proof}
(a) To distinguish the two composition operators we denote the one on $\FE_{\sigma}$ by $C_{\varphi}^{E}$. 
We note that 
\[
 \|C_{\varphi}^{E}(f)\|_{\sigma}
=\sup_{e^{\ast}\in B_{\|\cdot\|_{E^{\ast}}}}\|e^{\ast}\circ (f\circ \varphi)\|
= \sup_{e^{\ast}\in B_{\|\cdot\|_{E^{\ast}}}} \|C_{\varphi}(e^{\ast}\circ f)\|
\leq\|C_{\varphi}\|_{L(\F)}\|f\|_{\sigma}
\]
for all $f\in\FE_{\sigma}$, which implies the continuity of $C_{\varphi}^{E}$.

(b) First, we observe that the dual map $C_{\varphi}^{t}\colon \F^{\ast}\to\F^{\ast}$ leaves the $\|\cdot\|_{\F^{\ast}}$-closed 
linear subspace $\F_{\gamma}'$ invariant (see \prettyref{prop:dixmier_ng_mixed} (a)). Indeed, 
we have $C_{\varphi}^{t}(\delta_{x})=\delta_{\varphi(x)}$ for all $x\in\Omega$. Since the span of 
$\{\delta_{x}\;|\;x\in\Omega\}$ is $\|\cdot\|_{\F^{\ast}}$-dense in $\F_{\gamma}'$ by \prettyref{cor:siB_linearisation}, 
\prettyref{rem:point_eval_in_predual} and \cite[Proposition 2.7, p.~1596]{kruse2024}, we obtain that $C_{\varphi}^{t}$ leaves 
$\F_{\gamma}'$ invariant. Second, we claim that 
\[
C_{\varphi}^{E}=\chi\circ W_{\varphi}\circ \chi^{-1}
\]
where $\chi$ is the isometric isomorphism from \eqref{eq:si_lin_vec_valued} and 
\[
W_{\varphi}\colon L(\F_{\gamma}',E)\to L(\F_{\gamma}',E),\;u\mapsto \id_{E}\circ u\circ (C_{\varphi}^{t})_{\mid \F_{\gamma}'},
\]
with the identity map $\id_{E}$ on $E$. Indeed, we have by \eqref{eq:si_lin_vec_valued} that
\begin{align*}
  (\chi\circ W_{\varphi}\circ \chi^{-1})(f)(x)
&=\chi\bigl(\chi^{-1}(f)\circ (C_{\varphi}^{t})_{\mid \F_{\gamma}'}\bigr)(x)
 =\bigl(\chi^{-1}(f)\circ (C_{\varphi}^{t})_{\mid \F_{\gamma}'}\bigr)(\delta_{x})\\
&=\chi^{-1}(f)(\delta_{\varphi(x)})
 =\chi(\chi^{-1}(f))(\varphi(x))
 =f(\varphi(x))
 =C_{\varphi}^{E}(f)(x)
\end{align*}
for all $f\in\FE_{\sigma}$ and $x\in\Omega$. Since $\id_{E}$ is weakly compact 
by \cite[Proposition 23.25, p.~272]{meisevogt1997} due to the reflexivity of $E$ 
and $(C_{\varphi}^{t})_{\mid \F_{\gamma}'}$ is compact as it is the restriction to an invariant closed subspace of $\F^{\ast}$
of the compact operator $C_{\varphi}^{t}$ by \cite[Schauder's theorem 15.3, p.~141]{meisevogt1997}, 
we get that $W_{\varphi}$ is weakly compact by \cite[Theorem 2.9, p.~100]{saksman1992}. 
We conclude that $C_{\varphi}^{E}=\chi\circ W_{\varphi}\circ \chi^{-1}$ is weakly compact. 
\end{proof}

Let $\varphi\colon\D\to\D$ be holomorphic and $\mathcal{B}_{\alpha}\coloneqq\mathcal{B}v_{\alpha}(\D)$ 
with $v_{\alpha}(z)\coloneqq (1-|z|^{2})^{\alpha}$ for all $z\in\D$ and some $\alpha>0$. 
Necessary and sufficient conditions such that $C_{\varphi}\in L(\mathcal{B}_{\alpha})$ resp.~$C_{\varphi}$ is compact 
are given in \cite[Theorems 2.1, 3.1, p.~193, 198--199]{ohno2003}.

\section{Isometric equivalence and uniqueness of isometric preduals}
\label{sect:isom_equivalence_predual}

First, we prove that the isometric Banach preduals coming from two strong isometric Banach linearisations of the same Banach space 
$\F$ are isometrically equivalent. Second, we show how to avoid the assumption on linear independence 
in \prettyref{prop:strongly_unique_isom_predual} (b) and \prettyref{cor:unique_isom_predual} (b).

\begin{thm}\label{thm:isometric_sB_linearisation_in_equivalence_class}
Let $(\F,\|\cdot\|,\tau)$ be a semi-reflexive Saks space of $\K$-valued functions on a non-empty set $\Omega$ such that 
$\tau_{\operatorname{p}}\leq\tau$ and $(Z,\widetilde{T})$ an isometric Banach predual of $(\F,\|\cdot\|)$. 
Then the following assertions are equivalent.
\begin{enumerate}
\item[(a)] There exists $\widetilde{\delta}\in\mathcal{F}(\Omega,Z)_{\sigma}$ such that $(\widetilde{\delta},Z,\widetilde{T})$ 
is a strong isometric Banach linearisation of $\F$.
\item[(b)] $(Z,\widetilde{T})$ and $(\F_{\gamma}',\mathcal{I})$ are isometrically equivalent. 
\end{enumerate}
\end{thm}
\begin{proof}
(a)$\Rightarrow$(b) Due to \prettyref{cor:siB_linearisation}, \prettyref{rem:point_eval_in_predual}, 
\cite[Proposition 5.1 (a), p.~968]{kruse2025} and \cite[Proposition 5.2, p.~696]{kruse2025} we have that 
$(Z,\widetilde{T})$ and $(\F_{\gamma}',\mathcal{I})$ are isometrically equivalent.

(b)$\Rightarrow$(a) This implication follows from \cite[Proposition 5.3, p.~970]{kruse2025}.
\end{proof}

For the proof of our next two results we only need to adjust the proofs of \cite[Corollaries 5.10, 5.11, p.~973]{kruse2025} 
to the isometric setting.

\begin{cor}\label{cor:strongly_unique_isom_predual_without_lin_ind}
Let $(\F,\|\cdot\|,\tau)$ be a semi-reflexive Saks space of $\K$-valued functions on a non-empty set $\Omega$ such that 
$\tau_{\operatorname{p}}\leq\tau$. 
Then the following assertions are equivalent. 
\begin{enumerate}
\item[(a)] $(\F,\|\cdot\|)$ has a strongly unique isometric Banach predual. 
\item[(b)] For every isometric Banach predual $(Z,\varphi)$ of $(\F,\|\cdot\|)$ and every $x\in\Omega$ there is a (unique) $z_{x}\in Z$ with $\delta_{x}=\varphi(\cdot)(z_{x})$.
\end{enumerate}
\end{cor}
\begin{proof}
(a)$\Rightarrow$(b) This implication follows from \prettyref{prop:strongly_unique_isom_predual} since 
the triple $(\Delta,\F_{\gamma}',\mathcal{I})$ is a strong isometric Banach linearisation 
by the proof of \prettyref{thm:existence_siB_linearisation} and 
$\mathcal{I}(f)(\Delta(x))=\mathcal{I}(f)(\delta_{x})=\delta_{x}(f)$ for all $f\in\F$ and $x\in\Omega$. 

(b)$\Rightarrow$(a) Let $(Z,\varphi)$ be an isometric Banach predual of $(\F,\tau)$.     
Due to \prettyref{rem:semireflexive_pre_Saks_already_Saks} (e) $B_{\|\cdot\|}$ is $\sigma_{\varphi}(\F,Z)$-compact and 
$(\F,\|\cdot\|,\sigma_{\varphi}(\F,Z))$ a semi-Montel Saks space. 
By assumption for every $x\in\Omega$ there is $z_{x}\in Z$ with $\delta_{x}=\varphi(\cdot)(z_{x})$ and 
thus $\delta_{x}\in \F_{\gamma_{\varphi}}'$ by \cite[Proposition 3.21 (c), p.~1605]{kruse2024} and 
\cite[Proposition 3.16, p.~1602]{kruse2024} where $\gamma_{\varphi}\coloneqq \gamma(\|\cdot\|,\sigma_{\varphi}(\F,Z))$. 
Hence $(\Delta,\F_{\gamma_{\varphi}}',\mathcal{I}_{\varphi})$ is a strong isometric Banach linearisation of $\F$ 
by \prettyref{cor:siB_linearisation} with the evaluation map
\[
\mathcal{I}_{\varphi}\colon(\F,\|\cdot\|)\to (\F_{\gamma_{\varphi}}',\|\cdot\|_{\F_{\gamma_{\varphi}}'})^{\ast},\; 
f\longmapsto [f'\mapsto f'(f)].
\]
From \prettyref{thm:isometric_sB_linearisation_in_equivalence_class} we deduce that 
$(\F_{\gamma_{\varphi}}',\mathcal{I}_{\varphi})$ and $(\F_{\gamma}',\mathcal{I})$ are isometrically equivalent. 
Due to \cite[Proposition 3.21 (b), p.~1605]{kruse2024} and \prettyref{prop:equiv_implies_isom_equiv} we also know that 
$(\F_{\gamma_{\varphi}}',\mathcal{I}_{\varphi})$ and $(Z,\varphi)$ are isometrically equivalent. This implies that 
$(Z,\varphi)$ and $(\F_{\gamma}',\mathcal{I})$ are isometrically equivalent. Therefore $(\F,\|\cdot\|)$ has a strongly unique 
isometric Banach predual by \prettyref{prop:isom_strongly_unique_equivalent}.
\end{proof}

\begin{cor}\label{cor:unique_isom_predual_without_lin_ind}
Let $(\F,\|\cdot\|,\tau)$ be a semi-reflexive Saks space of $\K$-valued functions on a non-empty set $\Omega$ such that 
$\tau_{\operatorname{p}}\leq\tau$. 
Then the following assertions are equivalent. 
\begin{enumerate}
\item[(a)] $(\F,\|\cdot\|)$ has a unique isometric Banach predual. 
\item[(b)] For every isometric Banach predual $(Z,\varphi)$ of $(\F,\|\cdot\|)$ there is 
an isometric isomorphism $\psi\colon(\F,\|\cdot\|)\to Z^{\ast}$ such that for every $x\in\Omega$ there is a (unique) $z_{x}\in Z$ with $\delta_{x}=\psi(\cdot)(z_{x})$.
\end{enumerate}
\end{cor}
\begin{proof}
This statement follows from \prettyref{cor:unique_isom_predual} and the 
proof of \prettyref{cor:strongly_unique_isom_predual_without_lin_ind} with $\varphi$ replaced by $\psi$ and 
\prettyref{prop:isom_unique_equivalent} instead of \prettyref{prop:isom_strongly_unique_equivalent} in the end. 
\end{proof}

As an application of \prettyref{cor:strongly_unique_isom_predual_without_lin_ind} we consider the following simple example, 
which is \cite[Example 2.1, p.~413]{davidson2011}.

\begin{exa}\label{ex:linfty_str_unique_isom_Bpredual}
We show that the Banach space $(\ell^{\infty},\|\cdot\|_{\infty})$ of complex bounded sequences on 
$\N$ equipped with the supremum norm has a strongly unique isometric Banach predual. 
Due to \prettyref{ex:subspace_cont_mixed} (i) the triple $(\ell^{\infty},\|\cdot\|_{\infty},\tau_{\operatorname{co}})$ is 
a semi-Montel Saks space and $\tau_{\operatorname{p}}\leq\tau_{\operatorname{co}}$ (the space $\N$ equipped with the metric induced by the absolute value is discrete).
Let $(Z,\varphi)$ be an isometric Banach predual of $(\ell^{\infty},\|\cdot\|_{\infty})$ and $n\in\N$. 
As in \cite[Example 2.1, p.~413]{davidson2011} we observe that 
\begin{align*}
 \ker(\delta_{n})\cap B_{\|\cdot\|_{\infty}}
&=\{x\in\ell^{\infty}\;|\;x_{n}=0\text{ and }\|x\|_{\infty}=\sup_{k\in\N}|x_{k}|\leq 1\}\\
&=\{x\in\ell^{\infty}\;|\;|x_{n}-1|\leq 1,\;|x_{n}+1|\leq 1\text{ and }\sup_{k\in\N,k\neq n}|x_{k}|\leq 1\}\\
&=B_{\|\cdot\|_{\infty}}(e_{n})\cap B_{\|\cdot\|_{\infty}}(-e_{n}).
\end{align*}
Due to \prettyref{rem:semireflexive_pre_Saks_already_Saks} (e) $B_{\|\cdot\|_{\infty}}(e_{n})$ and $B_{\|\cdot\|_{\infty}}(-e_{n})$ 
are $\sigma_{\varphi}(\ell^{\infty},Z)$-compact, which implies that $\ker(\delta_{n})\cap B_{\|\cdot\|_{\infty}}$ is 
$\sigma_{\varphi}(\ell^{\infty},Z)$-compact as well. This yields that for every $n\in\N$ there is $z_{n}\in Z$ such that 
$\delta_{n}=\varphi(\cdot)(z_{n})$ by \prettyref{rem:weakly_continuous}. We conclude that $(\ell^{\infty},\|\cdot\|_{\infty})$ has 
a strongly unique isometric Banach predual by \prettyref{cor:strongly_unique_isom_predual_without_lin_ind}.
\end{exa}

Even though $\ell^{\infty}$ has a strongly unique isometric Banach predual by \prettyref{ex:linfty_str_unique_isom_Bpredual}, 
it does not have a unique Banach predual by \cite[Example 6.4.3, p.~200]{dales2016}. Moreover, it follows 
from \prettyref{ex:subspace_cont_mixed} (i) that the spaces $\ell v(\N)$ have a strongly unique isometric Banach predual. 
For some of the other spaces from \prettyref{ex:subspace_cont_mixed} it is also known that they have a 
strongly unique isometric Banach predual. For instance $H^{\infty}=\mathcal{H}w(\D)$ with $w(z)\coloneqq 1$ for $z\in\D$ 
and $\mathrm{Lip}_{0}(\Omega)$, if $\Omega$ has a finite diameter or is a complete convex metric space, 
have a strongly unique isometric Banach predual by \cite[Theorem 1, p.~34]{ando1978} and 
\cite[Theorems 3.2, 3.3, p.~471--472]{weaver2018a}.\footnote{For $\mathrm{Lip}_{0}(\Omega)$ it seems to be an open problem 
again whether it has a strongly unique isometric Banach predual if $\Omega$ has a finite diameter or is a complete convex metric space, see \cite{weaver2018b}.}

\bibliography{biblio_linearisation_mixed_isometric}
\bibliographystyle{plainnat}
\end{document}